\documentclass{birkjour}
\usepackage{autobreak}
\allowdisplaybreaks  

\usepackage{amsmath,amssymb,epsfig,color, amsthm, mathrsfs}
\usepackage{wrapfig,lipsum,floatrow,sidecap,comment}
\usepackage{graphicx}
\usepackage{multirow}
\usepackage{hhline}
\usepackage{url}
\usepackage{soul}
\usepackage{ulem}

\newcommand\R{\mathbb{R}}

 \newtheorem{thm}{Theorem}[section]
 \newtheorem{cor}[thm]{Corollary}
 \newtheorem{lem}[thm]{Lemma}
 \newtheorem{prop}[thm]{Proposition}
 
 \theoremstyle{definition}
 \newtheorem{defn}[thm]{Definition}
 \theoremstyle{remark}

 \numberwithin{equation}{section}
\usepackage{tikz - cd}
\usepackage{amsmath}

\begin{document}


\title[An Example for birkjour]
{Uniqueness of the Short-Time Linear Canonical Transform Phase Retrieval }

\author{Yali Dong}

\address{%
School of Mathematics  \\
Zhengzhou University of Aeronautics\\
Zhengzhou \\
China}

\email{yldong2022@163.com}

\thanks{This work was supported by Nation Natural Science Foundation of China (Grant Nos. 11971348; 12071230; 12471131) and the Basic Research Projects of Key Scientific Research Projects Plan in Henan Higher Education Institutions (25ZX013)}
\author{Rui Liu}
\address{School of Mathematical Sciences and LMPC\br
Nankai University \br
Tianjin \br
China}
\email{ruiliu@nankai.edu.cn}

\author{Heying Wang}
\address{School of Mathematical Sciences and LMPC\br
Nankai University \br
Tianjin \br
China}
\email{heyingwang@mail.nankai.edu.cn}


\date{January 1, 2025}

\begin{abstract}
In this paper, we focus on the problem of phase retrieval from intensity measurements of the Short-Time Linear Canonical Transform (STLCT). Specifically, we show that the STLCT allows for the unique recovery of any square-integrable function through phaseless STLCT sampling on rectangular square-root lattices. When turning to the uniform lattices, we establish counterexamples about the STLCT phase retrieval problems in $L^2(\R).$ 
Nevertheless, for functions in band-limited function spaces, phase retrieval results on uniform lattices can still be accomplished. 
\end{abstract}

\maketitle

\section{Introduction}\ 

Generally speaking, the phase retrieval problem refers to the problem of recovering a signal from phaseless linear measurements. 
With the applicative relations with the Fourier transform, phase retrieval problems arise in various areas of natural sciences, ranging from signal processing to quantum mechanics. The related problems also attract plenty of attention in mathematics recently, which leads to various mathematical theories emerging. 

The linear canonical transform (LCT), first mentioned in the 1970s, is an integral transform characterized by four parameters \(a, b,c, d\). It serves as a generalization of the Fourier transform, the fractional Fourier transform, the Fresnel transform, and the scaling operations. Under the Bargmann transform, Dong and Zhu show in \cite{DZ} that the LCT is unitarily equivalent to some two-parameter family of integral operators. The LCT was applied in optics in the early years, where it can represent a multitude of optical systems. Nowadays, it also applies to radar system analysis, filter design, pattern recognition, phase retrieval, and various other fields.

Chen and Qu study phase retrieval from linear canonical transforms (LCT) in \cite{CY}, showing that compactly supported functions can be determined up to global phase or conjugate reflection using magnitudes of multiple LCT measurements. They construct explicit masks and analyze the reduction of ambiguity, extending the results from fractional Fourier transforms to the broader LCT framework.

To get some further analysis about the localized information of signals, the Short-Time Fourier Transform (STFT) has been formulated as the Fourier transform with a window function. Consequently, phase retrieval problems related to the STFT have been intensively studied in recent years. Grohs et al. established the discretization barriers in \cite{LG}, which reveals that if $\Lambda$ is a lattice, then functions in $L^2(\mathbb{R}^d)$ cannot be uniquely determined by samples of the form of $|V_g f(\Lambda)|$, regardless of the lattice choice $\Lambda$ or the window function. 
Moreover, Alaifari1 and Wellershoff \cite{RM} demonstrate that for any lattice, one can construct functions in \(L^{2}(\mathbb{R})\), which do not agree up to global phase, but whose magnitudes sampled on the lattice agree under the STFT with the Gabor window function. 
Additionally, 
Wellershoff \cite{W} shows that for functions in general bandlimited function spaces \(L^{p}([-B, B])\),  phase recovery can still be accomplished on a uniform lattice.
Recently, driven by this discretization barrier, Grohs et al. in \cite{FCM} centers around the initiation of a novel sampling scheme which allows for unique recovery of any
square-integrable function via phaseless STFT-sampling.  For more general nonuniform sampling sets. For more informations, see \cite{K,  W,  FCM, LG, PL, RM}.

The short-time linear canonical transform (STLCT) is replacing the Fourier kernel with the LCT kernel in the STFT definition, enabling local time-frequency analysis with high resolution and cross-term suppression. It has been extensively applied in speech, acoustics, and many other signal processing domains.

Li and Zhang \cite{Z, ZL} investigate phase retrieval using STLCT, establishing uniqueness results for square-integrable functions, nonseparable real signals, and complex bandlimited signals in Paley–Wiener spaces and cardinal B-spline spaces. They show that STLCT magnitude measurements, under some mild conditions on window functions, can uniquely determine signals up to global phase.

Inspired by previous research, we focus on the problem of phase retrieval from intensity measurements of the Short-Time Linear Canonical Transform (STLCT). Specifically, we demonstrate that square-root lattices enable the unique recovery of any square-integrable function via phaseless STFT-sampling. However, it fails to do so on a uniform lattice, as we have found counterexamples. Nevertheless, for functions in bandlimited function spaces, phase recovery can still be accomplished on a uniform lattice. 

In this article, we first focus on the phase retrieval of the STLCT. We prove that functions in $L^2(\R)$ can be uniquely determined by the STLCT magnitudes sampled on rectangular square-root lattices up to a global phase, which can be stated as the following theorem. 


\begin{thm}
Let \(m, n \in \mathbb{R}_{>0}^{d'}\), and let \(0 \neq \varphi \in \mathcal{O}_{m}^{n}(\mathbb{C}^{d'})\) be a window function. Suppose that \(\Lambda = A(\sqrt{\mathbb{Z}})^{2d'}\) is a rectangular square-root lattice such that the generating matrix \(A = \mathrm{diag}(\tau_{1}, \ldots, \tau_{d'}, v_{1}, \ldots, v_{d'})\) with \(\tau, v \in \mathbb{R}_{>0}^{d'}\) satisfies
\[
\tau_{j} < \frac{1}{\sqrt{2n_{j}}e}, \quad v_{j} < b\sqrt{\frac{2m_{j}}{e}}, \quad j \in \{1, \ldots, d'\}.
\]
Then the following statements are equivalent for every \(f, h \in L^{2}(\mathbb{R}^{d'})\):
\[
(1) \ \left|V^{(A)}_{\varphi} f(\lambda)\right| = \left|V^{(A)}_{\varphi} h(\lambda)\right| \text{ for every } \lambda \in \Lambda, \quad (2) \ f \sim h.
\]
\end{thm}

Moreover, we show that for any lattice in the form \(m\mathbb{Z}\times \mathbb{R}\), one can
construct functions in \(L^{2}(\mathbb{R})\) that do not agree up to the global phase, but whose Gabor transform magnitudes sampled on the lattice agree. This also notes that parallel lines and uniform lattices can not be used in STLCT phase retrieval problems in $L^2(\R).$ I.e., we have the following theorem. 

\begin{thm}
Let \(f_{\pm} = (1 \pm i) M_{u}^{(A)}\varphi + (1 \mp i) M_{-u}^{(A)}\varphi\in L^{2}(\mathbb{R})\) with \(\varphi\) is a Gaussian window function,  then it holds that $f_{+}$ and $f_{-}$ do not agree up to global phase and yet
\begin{equation*}
 | \mathcal{G}^{(A)} f_{+}|=|\mathcal{G}^{(A)} f_{-}|
\end{equation*}
\end{thm}

Finally, for functions in band-limited function spaces, phase recovery can still be accomplished on the uniform lattices, which can be concluded as the following theorem. 

\begin{thm}
Let \(p \in [1, \infty)\), \(B > 0\), and \(m \in \left(0, \frac{1}{4B}\right)\). Then, the following are equivalent for \(f, g \in L^{p}([-B, B])\):  
\begin{enumerate}
    \item [(1)] \(f = e^{i\alpha} g\) for some \(\alpha \in \mathbb{R}\); 
    \item [(2)] \(|\mathcal{G}^{(A)}f| = |\mathcal{G}^{(A)}g|\) on \(\mathbb{N} \times mb\mathbb{Z}\).
\end{enumerate}   
\end{thm}

The whole article is arranged as follows: In Section 2, we introduce the linear canonical transform, the short-time linear canonical transform, and some basic properties. Some basic results about the phase retrieval problem of the linear canonical transform are also claimed. In Section 3, we mainly focus on determining functions in $L^2(\R)$ by using the STLCT sampled on rectangular square-root lattices up to a global phase. In Section 4, we first deliver the counterexample of the STLCT phase retrieval problem in $L^2(\R)$ sampled on parallel lines. The result of the STLCT phase retrieval problem of band-limited functions is also interpreted in this section. Finally, we add some additional details about the STLCT phase retrieval problem in $L^2(\R).$ 

\section{Preliminaries}

\subsection{Linear canonical transform}\ 

Firstly, we recall some related definitions of the linear canonical transform. 

\begin{defn}
A generalized modulation operator \(M_{\mu}^{(A)}\) characterized by a matrix
\( A = \begin{pmatrix} a & b \\ c & d \end{pmatrix} \), \( \det(A) = 1 \), with respect to variable \( t \) and effect on \( f \) is defined as
\[
M_{\mu}^{(A)} f(t) := e^{-i\left( \frac{a}{2b}t^2 - \frac{1}{b}t\mu + \frac{d}{2b}\mu^2 \right)} f(t).
\quad (2.1.1)
\]
\end{defn}
When $A=(0, 1, -1, 0)$, $M_{\mu}^{(A)} f(t)=M_{\mu}f(t)$.

The linear canonical transform (LCT) with parameter \( A = \begin{pmatrix} a & b \\ c & d \end{pmatrix} \) of a signal \( f(t) \) is defined
\[
L_A (f)(\mu) :=
\begin{cases}
\displaystyle\int_{\mathbb{R}^{d'}} K_A(t, \mu) f(t) \, dt, & b \neq 0; \\
\sqrt{d} e^{i \frac{c d}{2} \mu^2} f(d\mu), & b = 0.
\end{cases}
\quad (2.1.2)
\]
with
\[
K_A(t, \mu) := \frac{1}{\sqrt{ b}} e^{i\left( \frac{a}{2b} t^2 - \frac{1}{b} t \mu + \frac{d}{2b} \mu^2 \right)}, \quad b \neq 0. 
\]
where \( a, b, c, d \) are real or complex numbers satisfying \( ad - bc = 1 \), i.e., \( \det(A) = 1 \). \(L_{A}\) is a unitary operator on  
\(L^2(\mathbb{R}^{d'})\).
\[
L_{A}f(\mu)=\frac{1}{\sqrt{b}}e^{i\frac{d\mu^{2}}{2b}}\mathcal{F}(fe^{\frac{ia}{2b}t^{2}})(\frac{\mu}{b}), \quad (2. 1.3)
\]
Especially, when \(A=(0, 1, -1, 0)\), $L_{A}=\mathcal{F}$. 
\[
\mathcal{F}(f)(\mu)= \int_{\mathbb{R}}f(t)e^{-i\mu t}dt ;
\]
The extension of \(\mathcal{F}\) from 
\(L^1(\mathbb{R}^{d'}) \cap L^2(\mathbb{R}^{d'})\) to a unitary operator on \(L^2(\mathbb{R}^{d'})\) is carried out in the usual way. 

For any real numbers $a$ and $b,$ we define two unitary operators $T_{a}$,
$M_{b}$  and $D_{a}$ on $L^{2}(\mathbb{R}^{d'})$ as follows:
\[
T_{a}f(t)=f(t-a); \,\, M_{b}f(t)=e^{ibt}f(t), \]
where $T_{a}$ is known as the translation operator, $M_{b}$ is the modulation operator.

Some basic properties of the transforms above can be easily checked. We listed them in the following lemma. 

\begin{lem}
For every $\tau, \nu, x, \omega \in \mathbb{R}^{d'}$ and every $f,g \in L^{2}(\mathbb{R}^{d'})$, the operators defined above satisfy the relations
\begin{enumerate}
  \item [(1)] $\mathcal{F}T_{\tau}=M_{-\tau}\mathcal{F}$; $\mathcal{F}M_{\tau}=T_{\tau}\mathcal{F}$; $\mathcal{FR}=\mathcal{RF}$; $\mathcal{R}f(t)=f(-t)$;
  \item [(2)] $T_{\tau}M_{x}=e^{-ix\tau}T_{\tau}M_{x}$;
   \item [(3)] $V_{g}(T_{\tau}M_{\nu}f)(x, \omega)=e^{-2\pi i\tau\omega}V_{g}f(x-\tau, \omega-\nu)$
  \item [(4)] $\mathcal{R}L_{A}=L_{A}\mathcal{R}$, $T_{\tau}\mathcal{R}=\mathcal{R}T_{-\tau}$
  \item [(5)] $T_{\tau}M_{\mu}^{(A)} f(t)= (M_{\mu}^{(A)} f)(t-\tau)
  = e^{-\frac{i}{b}\mu\tau+\frac{ia}{b}t\tau-\frac{ia}{2b}\tau^2}M_{\mu}^{(A)}T_{\tau}f(t)$
  \item [(6)] $
L_{\mathbf{A}} f_{t, -\frac{a t}{b}}(\mu) =L_{\mathbf{A}}M_{-\frac{a t}{b}}T_{t} f(\mu)= e^{i \left( \frac{d}{2b} \mu^2 - \frac{\mu t}{b} - \frac{a t^2}{2b} \right)}L_{\mathbf{A}} f(\mu).$
\end{enumerate}
\end{lem}

As is mentioned in the introduction, phase retrieval problems aim at recovering signals from the phaseless measurements, which can also be stated as determining signals from their magnitudes up to a global phase. The linear canonical transform \(L_{A}\) has a more generalized form of the
well-known (fractional) Fourier transform and a wide range of engineering applications, such as optics and quantum mechanics. As mentioned in the introduction, there are already some conclusions about phase retrieval problems of the Fourier transform. By observing that
\begin{align*}
L_{A}f(\mu)=\frac{1}{\sqrt{b}}e^{i\frac{d\mu^{2}}{2b}}\mathcal{F}(\tilde{f})(\frac{\mu}{b}), 
\end{align*}
we can expect to get some similar results about phase retrieval problems of the linear canonical transforms. 
Recall that in \cite{Y}, signals can be determined by their own magnitudes and their magnitudes under some transforms up to a global constant and a conjugate operator. 

\begin{lem}(\cite{Y}, Theorem 1)
Let \(D\) be one of the operators: \(\frac{d}{dx}\), \(\gamma\) with \(q \in \mathbb{R}\) and \(|q| < 1\), or \(\delta\) with \(b \in \mathbb{R}\). Suppose that \(f\) and \(g\) are entire functions of genus \(\leq h\) such that
\[
|f(x)| = |g(x)| \quad \text{and} \quad |Df(x)| = |Dg(x)| \quad \forall x \in \mathbb{R}.
\]
for every real $x$. Then:
\begin{enumerate}
    \item [(a)] In the cases \(D = \frac{d}{dx}\), \(D = \gamma\) either \(f = cg\) or \(f = cg_*\) for some constant \(c \in \mathbb{C}\) with \(|c| = 1\).
    \item [(b)] In the case \(D = \delta\), either \(f = Wg\) or \(f = Wg_*\), where \(W\) is a meromorphic function with period \(b\) is continuous unimodular on the real line, and $g_*(x) = \overline{g(\bar{x})}$.
\end{enumerate}
\end{lem}

Thus, we can consider the linear canonical phase retrieval problem of determining signals from the magnitudes of their linear canonical transforms with some multiplication operators. 

\begin{prop}
Let $\varphi$ be the Gaussian function, and let $\L_{A}$ be the linear canonical transform with the matrix $A$. Then any signal $f\in L^{2}(\mathbb{R})$ can be determined up to a global phase by $|L_{A}Pf|, |L_{A}QPf|, ~and~ |L_{A}(I+Q)P f|$, where $Qf(t)=tf(t)$, \(P(t)=\varphi(t)\). 

If we consider determining functions up to a global phase and the conjunction, i.e. either \(f = cg\) or \(f = cg_*\) holds with \(|c| = 1\), any signal $f\in L^{2}(\mathbb{R})$ can be determined by 
\begin{enumerate}
\item [(1)] $|L_{A}Pf|, |L_{A}P_{1}f| ~and~ |L_{A}P_{2}f|$,  Here, $P_{1}(t)=\sin(a\pi t)\varphi(t)$, $P_{2}f(t)=\sin(b\pi t)\varphi(t)$, $a,b>0$ are such that $\frac{a}{b}\not \in\mathbb{Q}$.
\item [(2)] \(|L_{A}Pf|, |L_{A}\mathbf{D_1}f|\),  where \(\mathbf{D_1}(t)=D_{a_1}\varphi(t)-a_1^{\frac{1}{2}}\varphi(t)\) with $D_a
f(t)=a^{-\frac{1}{2}}f(\frac{t}{a})$ denoting the dilation operator.
\end{enumerate}
\end{prop}
\begin{proof}
Firstly, by observing that $L_{A}f(\mu)=\frac{1}{\sqrt{b}}e^{i\frac{d\mu^{2}}{2b}}\mathcal{F}(\tilde{f})(\frac{\mu}{b}),$ we can gain that for all $\mu \in \R,$
\begin{align}
|L_{A}QPf(\mu)|=|L_{A}QPg(\mu)| &\iff |\mathcal{F}QP\tilde{f}(\mu)|=|\mathcal{F}QP\tilde{g}(\mu)| \label{pythagorean} \\
|L_{A}(I+Q)Pf|=|\L_{A}(I+Q)Pg| &\iff |\mathcal{F}(I+Q)P\tilde{f}(\mu)|=|\mathcal{F}(I+Q)P\tilde{g}(\mu)| \label{pythagorean} \\
|\mathcal{C}_{A}Pf(\mu)|=|\mathcal{C}_{A}Pg(\mu)| &\iff |F(\mu)|=|G(\mu)| \label{pythagorean} \\
|L_{A}P_{1}f(\mu)|=|L_{A}P_{1}g(\mu)| &\iff |\mathcal{F}P_{1}\tilde{f}|=|\mathcal{F}P_{1}\tilde{g}| \label{pythagorean} \\
|L_{A}\mathbf{D}f(\mu)|=|L_{A}\mathbf{D}g(\mu)| &\iff |\mathcal{F}\mathbf{D}\tilde{f}(\mu)|=|\mathcal{F}\mathbf{D}\tilde{g}(\mu)| \label{pythagorean}
\end{align}
where we denote that $F\equiv \mathcal{F}[\varphi\tilde{f}]$ and $G\equiv \mathcal{F}[\varphi \tilde{g}]$ with $\tilde{f}=fe^{\frac{ia}{2b}\cdot^{2}}$ and $\tilde{g}=ge^{\frac{ia}{2b}\cdot^{2}}.$ Note that $F$ and $G$ are two entire functions extended from the two analytic functions defined on the real line respectively. 

Based on the equalities above, we can prove the theorem gradually.
\begin{enumerate}
\item[(1)] Recall that $i\mathcal{D}\mathcal{F}=\mathcal{F}Q$, where $\mathcal{D}$ is the differential operator.   
Then $|\mathcal{F}QP\tilde{f}(\mu)|=|\mathcal{F}QP\tilde{g}(\mu)|$ implies that $|F'|=|G'|$.
$|\mathcal{F}(I+Q)P\tilde{f})(\mu)|=|\mathcal{F}(I+Q)P\tilde{g})(\mu)|$ implies that
\begin{equation*}
 |F+iF'|=|G+iG'|.
\end{equation*}

Therefore,
\begin{equation*}
  |F|=|G|, ~~~ |F'|=|G'|, ~~~ |F+iF'|=|G+iG'|.
\end{equation*}
Since $|F+iF'|^{2}=|F|^{2}+|F'|^{2}-2Im{F'\overline{F}}$ and $Re{F'\overline{F}}=\frac{1}{2}(|F|^{2})'$,
we get $|F|=|G|$ and $F'\overline{F}=G'\overline{G}$.
According to \cite{I} Lemma 2.1, every complex-valued analytic function on \( \mathbb{R}\) can be uniquely determined by $|F|^2$ and $F'\bar{F}$ up to a unimodular constant, which implies that $G=\lambda F$ with $|\lambda|=1$. 

\item[(2)] As $\mathcal{F}M_{v}=T_{v}\mathcal{F}$, $\sin\alpha=\frac{e^{i\alpha}-e^{-i\alpha}}{2i}$ and $P_{1}(t)=\sin(a\pi t)\varphi(t)$. 
With $(3.4),$ a straightforward calculation shows that
\begin{align*}
   &  |\mathcal{F}P_{1}\tilde{f}|=|\mathcal{F}\frac{e^{ia\pi t}-e^{-ia\pi t}}{2i}\tilde{f}|\\
   & =|\frac{1}{2i}(T_{\frac{a}{2}}-T_{\frac{-a}{2}})\mathcal{F}\tilde{f}|\\
   &=\frac{1}{2}|F(\mu-\frac{a}{2})-F(\mu+\frac{a}{2})|
\end{align*}
Then we have
\begin{equation*}
  |\mathcal{C}_{A}P_{1}f|=|\mathcal{C}_{A}P_{1}g|
\end{equation*}
if and only if
\begin{equation*}
  |F(\mu)-F(\mu-a)|=|G(\mu)-G(\mu-a)|.
\end{equation*}
Similarly, $|L\mathcal{C}_{A}P_{2}f|=|L_{A}P_{2}g|$ if and only if $|F(\mu)-F(\mu-b)|=|G(\mu)-G(\mu-b)|$.

Applying \cite{Y} Theorem 1 twice, we get that there exist two periodic functions $W_{a}$ and $W_{b}$ with respective periods $a$ and $b$, and $|W_{a}(x)|=|W_{b}(x)|=1$, such that $G=W_{a}F=W_{b}F$. Since $\frac{a}{b}\not \in\mathbb{Q}$, $\{ak+bl, k,l\in\mathbb{Z}\}$ is dense in $\mathbb{R}$. We get that $W_{a}$ is a constant of modulus one.
By the inverse transform, we have $f=\lambda g$, where $|\lambda|=1$. 

\item[(3)] Combining $(3.6)$ with the properties that $\mathbf{D}f(t)=D_{a}f(t)-a^{\frac{1}{2}}f(t)$ and $\mathcal{F}D_{a}=D_{\frac{1}{a}}\mathcal{F}$, a straightforward calculations shows that
\begin{align*}
& |\mathcal{F}[\mathbf{D}\tilde{f}]|=|\mathcal{F}[\mathbf{D}\tilde{f}(t)-a^{\frac{1}{2}}\tilde{f}(t)]|\\
   & =|D_{\frac{1}{a}}\mathcal{F}\tilde{f}-a^{\frac{1}{2}}\mathcal{F}\tilde{f}(t)|\\
   &=|a^{\frac{1}{2}}||F(a\mu)-F(\mu)|
\end{align*}
Then we have
\begin{equation*}
|\mathcal{C}_{A}\mathbf{D}f|=|\mathcal{C}_{A}\mathbf{D}g|
\end{equation*}
if and only if
\begin{equation*}
  |F(a\mu)-F(\mu)|=|G(a\mu)-G(\mu)|.
\end{equation*}

Applying the Theorem 1 in \cite{Y}, we get that $F=cG$, which yields the desired statement.
    
\end{enumerate}
\end{proof}

\subsection{Short-time linear canonical transform}\ 

In this subsection, we introduce the definition of the short-time linear canonical transform and some related preliminaries. 

\begin{defn}
For $b\neq0$, $A=(a,b,c,d)$, the short-time linear canonical transform (STLCT) of a function $f$ with respect to $g$
is defined by
\begin{align*}
V_{g}^{(A)}f(x, \mu) :
&= \frac{1}{\sqrt{b}} \int_{\mathbb{R}^{d'}} \overline{M_{\mu}^{(A)}T_x g(t)} f(t)dt \\
&= \int_{\mathbb{R}^{d'}} K_A(t, \mu) \overline{T_x g(t)} f(t)dt\\
&= L_{A}(f\overline{T_x g})(\mu) \quad (2.2.1)
\end{align*}
is a unitary operator on $L^{2}(\mathbb{R}^{d'})$.
\end{defn}

When \(A = (0, 1, -1, 0)\), the STLCT degenerates to the short time Fourier transform (STFT) of the function \(f\) with respect to the window function \(g\), denoted as:
\[
V_{g} f(x, \mu) := \int_{\mathbb{R}^{d'}} f(t) \overline{T_{x} g(t)} e^{-i t \mu}=\mathcal{F}(f\overline{T_{x} g})(\mu) \, dt \quad \text{for} \ x, \mu \in \mathbb{R}^{d'}. \quad (2.2.2)
\]



Similar to the communications of short-time Fourier transform and the multiplications, communicative relationships between the short-time linear canonical transform and the multiplications are established in \cite{WLCT} as the following lemma.

\begin{lem}(\cite{WLCT})For every $\mu, u\in \mathbb{R}^{d'}$ and every $f,g \in L^{2}(\mathbb{R}^{d'})$, the operators defined above satisfy the relations
\[L_{A_{1}}\left(M_{\mu}^{\left(A_{2}\right)} f\right)(u)=M_{\mu}^{(D)} T_{\left(b_{1} / b_{2}\right) \mu} L_{C}(f)(u),\]
where \(C=(a_{1}-\frac{a_{2}}{b_{2}} b_{1}, b_{1}, c_{1}-\frac{a_{2}}{b_{2}}d_{1}, d_{1})\) and \(D=\) \((0, \frac{b_{2}}{d_{1}}, -\frac{d_{1}}{b_{2}}, \frac{d_{2}}{d_{1}}+\frac{b_{1}}{b_{2}})\). Especially, when $A_{1}=A_{2}$, 
\[
L_{A}\left(M_{\mu}^{\left(A\right)} f\right)=M_{\mu}^{(D)} T_{\mu} L_{C}(f),
\]
where \(C=(0, b, c-\frac{a}{b}d, d)=(0, b, -\frac{1}{b}, d)\) and \(D=\) \((0, \frac{b}{d}, -\frac{d}{b}, 2)\). when $d=0$, $M_{\mu}^{(D)}=1$.
$A=C$ if and only if $a=0$.
\end{lem}

\begin{defn}(\cite{L})
We define the convolution \( *_{A} \) of two functions \( f \) and \( g \) associated with the STLCT as
\[
\left( f *_{A} g \right)(t) = \frac{\overline{\lambda}_{A}(t)}{\sqrt{ b}} \left( f_{A} * g_{A} \right)(t),
\]
where \( \lambda_{A}(t) = \exp\left( \frac{i a t^{2}}{2 b} \right) \) is the chirp-modulation function, \( \overline{\lambda}_{A}(t) = \exp\left( \frac{-i a t^{2}}{2 b} \right) \), and \( \tilde{f}(t) = \lambda_{A}(t) f(t) = \exp\left( \frac{i a t^{2}}{2 b} \right) f(t) \).
\end{defn}

Similar to the convolution is taken to point-wise multiplication under the Fourier transform. We extend the result to the linear canonical transform.

\begin{lem}For \( f, g \in L^{2}(\mathbb{R}^{d'})\), 
under the linear canonical transform, the convolution is transformed into point-wise multiplication, and it is given by
\[
L_{A} (f *_A g) (\mu)=e^{-\frac{id}{2b}\mu^2} L_A f(\mu) L_A g(\mu).
\]
\end{lem}
\begin{proof}The proof is a straight-forward computation
\begin{align*}
L_{A} (f *_A g) (\mu) &= \frac{1}{ b} \int_{\mathbb{R}} e^{i(\frac{a}{2b}t^2-\frac{1}{b}t \mu + \frac{d}{2b} \mu^2)} e^{-\frac{ia t^2}{2b}} \int_{\mathbb{R}} f(s) e^{\frac{ia}{2b}s^2} g(t-s) e^{\frac{ia}{2b}(t-s)^2} ds dt \\
&= \frac{1}{ b} \int_{\mathbb{R}} f(s) \int_{\mathbb{R}} e^{i (\frac{a}{2b}s^2 - \frac{1}{b} t \mu + \frac{d}{2b} \mu^2)} e^{\frac{ia}{2b}(t-s)^2} g(t-s) dt ds \\
&\stackrel{\text{v=t-s}}{=} \frac{1}{b} \int_{\mathbb{R}} f(s) e^{\frac{ia}{2b}s^2 - \frac{i}{b} s \mu} \int_{\mathbb{R}} e^{i (\frac{a}{2b} v^2 - \frac{1}{b} v \mu + \frac{d}{2b} \mu^2)} g(v) dv ds \\
&= \frac{1}{\sqrt{b}} e^{- \frac{id}{2b} \mu^2} L_A (g)(\mu)  \int_{\mathbb{R}} e^{i (\frac{a}{2b} s^2 - \frac{1}{b} s \mu + \frac{d}{2b} \mu^2)} f(s) ds  \\
&= e^{-\frac{id}{2b}\mu^2} L_A f(\mu) L_A g(\mu).
\end{align*}
\end{proof}

\section{Phase retrieval of short-time linear canonical transform}\ 

In this section, we investigate the recovery of square-integrable signals from discrete equidistant samples of their Gabor linear canonical transform (LCT) magnitudes. 
We focus primarily on the phase retrieval problem for the short-time linear canonical transform (STLCT) when rectangular square-root lattices serve as sampling sets. 
Before presenting our main results, we first introduce some relevant lemmas, which underpins the subsequent exploration.

Recall that if \(m, n \in \mathbb{R}_{>0}^{d'}\), then 
\( \mathcal{O}_{m}^{n}(\mathbb{C}^{d'})\) denotes the collection of all entire functions \(F \in  \mathcal{O}(\mathbb{C}^{d'})\) which satisfy the estimate 
\[
|F(x + i y)| \lesssim \prod_{j=1}^{d} e^{-m_{j} x_{j}^{2}} e^{n_{j} y_{j}^{2}}, \ x, y \in \mathbb{R}^{d'}.
\]
The function class \( \mathcal{O}_{m}^{n}(\mathbb{C}^{d'})\) is a linear subspace of \( \mathcal{O}(\mathbb{C}^{d'}) \cap L^{2}(\mathbb{R}^{d'})\) (the space of all entire functions which are square-integrable on \(\mathbb{R}^{d'}\) ), and if \(m < n\) (i.e., \(m_{j}<n_{j}\) for every \(j \in\{1,\cdots,d'\}\) ), then \( \mathcal{O}_{m}^{n}(\mathbb{C}^{d'})\) is dense in \(L^{2}(\mathbb{R}^{d'})\). Note that we will repeatedly identify entire functions with their restriction on \(\mathbb{R}^{d'}\), and vice versa, functions on \(\mathbb{R}^{d'}\) with their analytic extension on \(\mathbb{C}^{d'}\) (provided the extension exists).  Window functions which belong to the function space \( \mathcal{O}_{m}^{n}(\mathbb{C}^{d'})\) lead to phaseless sampling results from square-root lattices. 



Just as the Fourier transform rotates the time-frequency plane in the STFT, we extend this result to the short-time linear canonical transform (STLCT). Our findings demonstrate that even under more complex operators, a close connection exists between the time-frequency localization analysis of a signal and its spectral properties, this connection serves as a natural generalization of the time-frequency rotational relationship in the STFT.

\begin{prop} \label{formula}
For \(f, g \in L^{2}(\mathbb{R}^{d'})\), the following holds;
\[
V_{g}^{(A)}f(x, \mu)=e^{\frac{i\mu (d\mu-x)}{b}}V^{(B)}_{L_{C}(g)}L_{A}(f)(\mu, d\mu-x)
\]
where \(A=(a, b, c, d)\), \(B=(-d, b, -\frac{1}{b}, 0)\), \(C=(0, b, -\frac{1}{b}, d)\). Especially, when
\(A=(0, 1, -1, 0)\), this yields the Fundamental Identity of Time-Frequency Analysis
\[
V_{g}f(x, \mu)=e^{-i\mu x}V_{\widehat{g}}(\widehat{f})(\mu, -x)
\]
\end{prop}
\begin{proof}
The proof is a straight-forward computation and the fact \(L_{A}\) is a unitary operator on 
\(L^{2}(\mathbb{R}^{d'})\).
\begin{align*}
V_{g}^{(A)}f(x, \mu)
&= \frac{1}{\sqrt{ b}}\langle f, M_{\mu}^{(A)}T_x g \rangle \\
&\stackrel{\text{(5)}}{=}\frac{1}{\sqrt{ b}} \langle f, e^{\frac{1}{b}i\mu x-\frac{ia}{b}tx+\frac{ia}{2b}x^{2}}T_x M_{\mu}^{(A)}g \rangle \\
&=\frac{1}{\sqrt{ b}}e^{-\frac{1}{b}i\mu x-\frac{ia}{2b}x^{2}}\langle f, M_{-\frac{ax}{b}}T_x M_{\mu}^{(A)}g \rangle \\
&= \frac{1}{\sqrt{ b}}e^{-\frac{1}{b}i\mu x-\frac{ia}{2b}x^{2}} \langle L_{A}(f), L_{A}M_{-\frac{ax}{b}}T_x M_{\mu}^{(A)}g \rangle \\
&\stackrel{\text{(6)}}{=}\frac{1}{\sqrt{ b}}e^{-\frac{1}{b}i\mu x-\frac{ia}{2b}x^{2}} \langle L_{A}(f), e^{\frac{di}{2b}u^{2}-\frac{iux}{b}-\frac{ia}{2b}x^{2}}L_{A}M_{\mu}^{(A)}g \rangle \\
&=\frac{1}{\sqrt{ b}}e^{-\frac{i\mu x}{b}}\langle L_{A}(f), e^{\frac{di}{2b}u^{2}-\frac{iux}{b}}L_{A}M_{\mu}^{(A)}g \rangle\\
&\stackrel{\text{(lem 2.6)}}{=}\frac{1}{\sqrt{ b}}e^{-\frac{i\mu x}{b}}\langle L_{A}(f), e^{\frac{di}{2b}u^{2}-\frac{iux}{b}}M_{\mu}^{(D)}T_{\mu}L_{C}(g) \rangle\\
&=\frac{1}{\sqrt{ b}}e^{-\frac{i\mu x}{b}}\langle L_{A}(f), e^{\frac{di}{2b}u^{2}-\frac{ixu}{b}}e^{i(\frac{d}{b}u \mu-\frac{d}{b}\mu^{2})}T_{\mu}L_{C}(g) \rangle\\
&=\frac{1}{\sqrt{ b}}e^{-\frac{i\mu x}{b}}e^{\frac{id}{b}\mu^{2}}\langle L_{A}(f), M_{d\mu-x}^{(B)}T_{\mu}L_{C}(g) \rangle\\
&=e^{\frac{i\mu (d\mu-x)}{b}}V^{(B)}_{L_{C}(g)}L_{A}(f)(\mu, d\mu-x),
\end{align*}
The third equation is due to the fact that
$M_{\mu}^{(D)}f(u)=e^{-i(0-\frac{d}{b}u \mu+\frac{d}{b}\mu^{2})}f(u)=e^{i(\frac{d}{b}u \mu-\frac{d}{b}\mu^{2})}f(u)$,
where $B=(-d, b, -\frac{1}{b}, 0)$, \(C=(0, b, -\frac{1}{b}, d)\).
\end{proof}

For the main result of this section, we make a assertion on STLCT phase retrieval with windows in $\mathcal{O}_{m}^{n}(\mathbb{C}^{d'}).$ This analysis is made under the assumption that complete spectrograms are available.

\begin{prop}
Let \(m , n \in \mathbb{R}_{>0}^{d'}\), and let \(0 \neq \varphi \in  \mathcal{O}_{m}^{n}(\mathbb{C}^{d'}) \) be a window function. If \(f, h \in L^{2}(\mathbb{R}^{d'})\) are such that \(|V^{(A)}_{\varphi} f(x, \mu)| = |V^{(A)}_{\varphi} h(x, \mu)|\) for every \((x, \mu) \in \mathbb{R}^{2d'}\), then \(f \sim h\).
\end{prop}

This proposition can be proved by using the formula in Proposition \ref{formula}. Since we aim to consider the phase retrieval problem sampling by rectangular square-root lattices, we will not state the complete proof of this proposition until Section 4. In order to lead to the result, we need several lemmas about $\mathcal{O}_{m}^{n}(\mathbb{C}^{d'})$ and its relationship with the short-time linear canonical transform.



Before we proceed to the main theorem, we show the following uniqueness statement.
\begin{lem}(\cite{FCM} Proposition 2.5) 
Let \(n = (n_{1}, \ldots, n_{d'}) \in \mathbb{R}_{>0}^{d'}\). For each \(j \in \{1, \ldots, d'\}\), let  
\[
\Lambda_{j} := \left\{ \pm \lambda_{j}(k) : k \in \mathbb{N}_{0} \right\} \subseteq \mathbb{R},
\]  
where \(\lambda_{j}: \mathbb{N}_{0} \to \mathbb{R}_{>0}\) is an increasing function.  

If  
\[
\liminf_{k \to \infty} \frac{\lambda_{j}(k)}{\sqrt{k}} < \frac{1}{\sqrt{n_{j} e}}
\]  
for every \(j \in \{1, \ldots, \mathbb{d}\}\), then \(\Lambda := \Lambda_{1} \times \cdots \times \Lambda_{\mathbb{d}}\) is a uniqueness set for \(O^{n}(\mathbb{C}^{\mathbb{d}})\).  
\end{lem}

The Fourier transform exhibits a reciprocal-scaling property for functions in these analytic function spaces: it maps \( F \in \mathcal{O}_m^n(\mathbb{C}^{\mathbb{d}}) \) to \( \mathcal{O}_{\frac{1}{4n}}^{\frac{1}{4m}}(\mathbb{C}^{\mathbb{d}}) \), inversely relating the defining parameters \( m \) and \( n \) with a scaling factor of \( \frac{1}{4} \). 
Note that we only prove the one-dimensional case and the high-dimensional cases can be proved by the same process. 
\begin{lem}If \(F \in  \mathcal{O}_m^n(\mathbb{C})\), then \(\mathcal{F}F\in  \mathcal{O}_{\frac{1}{4n}}^{\frac{1}{4m}}(\mathbb{C})\).
\end{lem}
\begin{proof} 
For \(F \in O_m^n(\mathbb{C})\) with \(|F(x+iy)| \leq C e^{-m x^2} e^{n y^2}\), its Fourier transform is:  
\[
\mathcal{F}F(\omega) = \int_{\mathbb{R}} F(x) e^{- i \omega x} dx, \quad \omega \in \mathbb{R}.
\]  
By Cauchy's theorem, the analytic continuation to complex frequency \(\omega + i\eta\):  
\[
\mathcal{F}F(\omega + i\eta) = \int_{\mathbb{R}} F(x+iy) e^{- i (\omega+i\eta)(x+iy)} dx.
\]
For \(y \in \mathbb{R}\), we get the modulus estimate :  
\[
|\mathcal{F}F(\omega + i\eta)| \leq \int_{\mathbb{R}} |F(x+iy)| e^{ \eta x} e^{ \omega y} dx.
\]  
Since \(|F(x+iy)| \leq C e^{-m x^2} e^{b n^2}\), we have  
\[
|\mathcal{F}F(\omega + i\eta)| \leq C e^{n y^2 + \omega y} \int_{\mathbb{R}} e^{-m x^2 + \eta x} dx.
\]
Through the completion of the square in the exponential argument:
\[
-m x^2 + \eta x = -m \left(x - \frac{\eta}{2m}\right)^2 + \frac{\eta^2}{4m}.
\]  
Substituting this back into the modulus estimate, we obtain:
\[
|\mathcal{F}F(\omega + i\eta)| \lesssim  e^{n y^2 + \omega y + \frac{ \eta^2}{4m}}.
\]
The quadratic expression \(n y^2 +\omega y\) achieves its minimum at \(y = -\frac{\omega}{2n}\),
which gives the lower bound
\[n y^2 +\omega y\geq -\frac{\pi^2 \omega^2}{n}.\]  
Based on the above results, we can derive the final frequency-domain estimate:
\[
|\mathcal{F}F(\omega + i\eta)| \lesssim e^{-\frac{\omega^2}{4n} + \frac{\eta^2}{4m}}.
\]  
\end{proof}
The next lemma constitutes a key component in the proof of Theorem 1.1. It asserts, that the modulus squared of the short - time linear canonical transform, \( |V^{(A)}_{\varphi}f|^{2} \), extends from \( \mathbb{R}^{2d'} \) to an entire function belonging to \( \mathcal{O}^{A}(\mathbb{C}^{2d'}) \), provided that \( \varphi \in \mathcal{O}_{m}^{n}(\mathbb{C}^{d'}) \).

\begin{lem}
Let \(m, n \in \mathbb{R}_{>0}^{d'}\), let \(\varphi \in \mathcal{O}_{m}^{n}(\mathbb{C}^{d'})\), and let 
\[
A:=\left(2n_{1}, \ldots, 2n_{d'}, \frac{1}{2m_{1}b^{2}}, \ldots, \frac{1}{2m_{d}b^{2}}\right) \in \mathbb{R}_{>0}^{2 d'}. \tag{10}
\]
Then for every \(f \in L^{2}(\mathbb{R}^{d'})\), it holds that \(|V^{(A)}_{\varphi} f|^{2} \in \mathcal{O}^{A}(\mathbb{C}^{2 d'})\); i.e., \(|V^{(A)}_{\varphi} f|^{2}\) extends from \(\mathbb{R}^{2 d'}\) to an entire function belonging to the space \(\mathcal{O}^{A}(\mathbb{C}^{2 d'})\).
\end{lem}
\begin{proof}
Step 1: Let \(z=x+i \rho \in \mathbb{C}^{d'}\). It is clear that \(\tilde{\varphi}\in  \mathcal{O}_{m}^{n}(\mathbb{C}^{d'})\) if \(\varphi \in \mathcal{O}_{m}^{n}(\mathbb{C}^{d'})\). From the proof of Lemma 3.4 and the arithmetic that
\[
|L_{A}\varphi(\mu)|=|\frac{1}{\sqrt{b}}e^{i \frac{d}{2b} \mu^2} \mathcal{F}(\tilde{\varphi})(d\mu)|=
|\frac{1}{\sqrt{b}}\mathcal{F}(\tilde{\varphi})(\frac{\mu}{d})|,
\]
we get that \(L_{A}\varphi\in  \mathcal{O}_{\frac{1}{4nd}}^{\frac{1}{4md}}(\mathbb{C}^{d'})\) if \(\tilde{\varphi}\in  \mathcal{O}_{m}^{n}(\mathbb{C}^{d'})\). 
Thus, we can conclude that \(|\hat{\tilde{\varphi}}(x+iy)|\lesssim \prod_{j=1}^d e^{-\frac{1}{4n_{j}} x_{j}^{2}} e^{\frac{1}{4m_j} y_{j}^{2}}\), 
then the modulus of the inner product of an arbitrary function \(u \in L^{1}(\mathbb{R}^{d'})\) with a complex shift of \(\varphi_{t}\) by \(z = x + i \rho\) is therefore upper bounded by
\begin{equation*} 
\left|\left\langle u, T_{z} \varphi^{A}_{t}\right\rangle\right|  \lesssim  
\prod_{j=1}^{d} e^{-\frac{1}{8n_{_{j}}} t_{j}^{2}} e^{\frac{1}{2 m_{j}} \rho_{j}^{2}} \|u\|_{L^{1}(\mathbb{R}^{d'})}.
\end{equation*}
where \(
f^{A}_{\omega}(t):=(T_{\omega} \widehat{\tilde{f}}(t))\overline{\widehat{\tilde{f}}(\overline{t})},\) and \(\tilde{f}(t):=e^{i  \frac{a}{2b} t^{2}} f(t), \omega \in \mathbb{C}^{d'}. \)
Now replace in the previous estimate the function \( u \) by the tensor product \( f_t \), where \( f \in L^2(\mathbb{R}^{d'}) \). The Cauchy-Schwarz inequality shows that
\[
\left\| f^{A}_{t} \right\| _{L^{1}\left(\mathbb{R}^{d'}\right)} \leq \| \widehat{\tilde{f}} \| _{L^{2}\left(\mathbb{R}^{d'}\right)}^{2}=\| f \| _{L^{2}\left(\mathbb{R}^{d'}\right)}^{2}
\]
Consider the Fourier integral 
\[
F\left(z, z'\right):=\int_{\mathbb{R}^{d'}}\left\langle f_{t}, T_{z} \varphi^{A}_{t}\right\rangle e^{ i z' \cdot t} \, dt .
\]
For \(z' = v + i\delta \in \mathbb{C}^{d'}\) with \(v, \delta \in \mathbb{R}^{d'}\), the Fourier integral \(F\) satisfies the estimate
\[
\begin{aligned}
&\left|F\left(z, z'\right)\right| \leq \int_{\mathbb{R}^{d'}} \left|\left\langle f, T_{z} \varphi^{A}_{t}\right\rangle\right| \prod_{j=1}^{d} e^{- \delta_{j} t_{j}} \, dt \\
&\quad \leq C(m) \|f\|_{L^{2}(\mathbb{R}^{d'})}^{2}\prod_{j=1}^{d} e^{\frac{1}{2m_{j}} \rho_{j}^{2}} \prod_{j=1}^{d} e^{2n_{j} \delta_{j}^{2}} ,
\end{aligned}
\]
for some constant \(C(m)\) depending only on \(m\). Moreover, an application of Lemma 3.2 shows that \(F\) defines an entire function on \(\mathbb{C}^{2d'}\).

 Step 2: We apply the bounds derived in Step 1 to the modulus squared of the STLCT. To do so, by \cite{ZL}, we have
\[\mathcal{F} | V^{(A)}_\varphi f(\cdot, u)|^2 = \frac{1}{b} \left[ \left( \widehat{\tilde{f}} \cdot T_{\frac{u}{b}} \widehat{\tilde{\varphi}} \right) * \left( \widehat{\tilde{f}} \cdot T_{\frac{u}{b}} \widehat{\tilde{\varphi}} \right)^\# \right]\]
for fixed \( u \in \mathbb{R}^{d'} \). The convolution \(\left( \widehat{\tilde{f}} \cdot T_{\frac{u}{b}} \widehat{\tilde{\phi}} \right) * \left( \widehat{\tilde{f}} \cdot T_{\frac{u}{b}} \widehat{\tilde{\phi}} \right)^\# \) evaluated at \(s\in \mathbb{R}^{d'}\) is given by 
\begin{align*}
\left( \widehat{\tilde{f}} \cdot T_{\frac{u}{b}} \widehat{\tilde{\varphi}} \right) * \left( \widehat{\tilde{f}} \cdot T_{\frac{u}{b}} \widehat{\tilde{\varphi}} \right)^\#(s)
&=\int_{\mathbb{R}^{d'}} \widehat{\tilde{f}}(-(s - t))\overline{\widehat{\tilde{\varphi}}(-(s - t)-\frac{u}{b})\widehat{\tilde{f}}(t)}\widehat{\tilde{\varphi}}(t - \frac{u}{b})dt\\
&=\int_{\mathbb{R}^{d'}} f^{A}_{s}(t)\overline{T_{\frac{u}{b}}\varphi^{A}_{s}(t)}dt\\
&= \langle f^{A}_{s}, T_{\frac{u}{b}}\varphi^{A}_{s}\rangle_{L^{1}(\mathbb{R}^{d'})\times L^{\infty}(\mathbb{R}^{d'})}\\
&= \langle f^{A}_{s}, T_{\frac{u}{b}}\varphi^{A}_{s}\rangle.
\end{align*} Then we obtain that
\[
b\left|V_{\varphi}^{(A)} f(x, \omega)\right|^{2}=\mathcal{F}^{-1}\left(t \mapsto\left\langle f^{A}_{t}, T_{\frac{\omega}{b}} \varphi^{A}_{t}\right\rangle\right)(x) = F(\frac{\omega}{b}, x),
\]
for every \((x, \omega) \in \mathbb{R}^{2d'}\). According to Step 1, the function \(F\) extends from \(\mathbb{R}^{2d'}\) to an entire function of \(2d'\) complex variables and satisfies the growth estimate
\[
\left|F\left(\frac{z'}{b}, z\right)\right| \lesssim \prod_{j=1}^{d} e^{\frac{1}{2m_{j}b^{2}}|z'|^{2}} \prod_{j=1}^{d'} e^{\frac{1}{2n_{j}}|z|^{2}}.
\]
This implies that for \(A \in \mathbb{R}_{>0}^{2d'}\) given as in Equation, we have \(|V^{A}_{\varphi} f|^{2} \in O^{A}(\mathbb{C}^{2d'})\).
\end{proof}

With the preparations above, we are ready to prove Theorem 1.1.
\begin{thm}
Let \(m, n \in \mathbb{R}_{>0}^{d'}\), and let \(0 \neq \varphi \in \mathcal{O}_{m}^{n}(\mathbb{C}^{d'})\) be a window function. Suppose that \(\Lambda = A(\sqrt{\mathbb{Z}})^{2d'}\) is a rectangular square-root lattice such that the generating matrix \(A = \mathrm{diag}(\tau_{1}, \ldots, \tau_{d'}, v_{1}, \ldots, v_{d'})\) with \(\tau, v \in \mathbb{R}_{>0}^{d'}\) satisfies
\[
\tau_{j} < \frac{1}{\sqrt{2n_{j}}e}, \quad v_{j} < b\sqrt{\frac{2m_{j}}{e}}, \quad j \in \{1, \ldots, d'\}.
\]
Then the following statements are equivalent for every \(f, h \in L^{2}(\mathbb{R}^{d'})\):
\[
(1) \ \left|V^{(A)}_{\varphi} f(\lambda)\right| = \left|V^{(A)}_{\varphi} h(\lambda)\right| \text{ for every } \lambda \in \Lambda, \quad (2) \ f \sim h.
\]
\end{thm}
\begin{proof}
The fact that (2) implies (1) is trivial. It remains to show that (1) implies (2). To this end, define \(\Psi := M(\sqrt{\mathbb{Z}})^{d'}\) and \(\Gamma := N(\sqrt{\mathbb{Z}})^{d'}\) with \(M = \mathrm{diag}(\tau_{1}, \ldots, \tau_{d'})\) and \(N = \mathrm{diag}(v_{1}, \ldots, v_{d'})\), so that \(\Lambda = \Psi \times \Gamma\). The choice of \(\tau\) and \(v\), in conjunction with Proposition 2.5 in \cite{FCM}, if
\(\tau_{j}<\frac{1}{\sqrt{2n_{j}}e}\) and 
\(v_j<\frac{1}{\sqrt{\frac{1}{2m_{j}b^{2}}}e}=b\sqrt{\frac{2m_{j}}{e}}\),
then \(\Lambda\) is a uniqueness set for the space \(O^{A}(\mathbb{C}^{2d'})\) with 
\[
A:=\left(2n_{1}, \ldots, 2n_{d'}, \frac{1}{2m_{1}b^{2}}, \ldots, \frac{1}{2m_{d'}b^{2}}\right) \in \mathbb{R}_{>0}^{2 d'}. \tag{10}
\]
According to above Lemma, it holds that \(|V^{A}_{\varphi} f|^{2} \in O^{A}(\mathbb{C}^{2d'})\) and \(|V^{A}_{\varphi} h|^{2} \in O^{A}(\mathbb{C}^{2d'})\). Since \(\Lambda\) is a uniqueness set for \(O^{A}(\mathbb{C}^{2d'})\), it follows that \(|V^{A}_{\varphi} f|^{2} = |V^{A}_{\varphi} h|^{2}\), i.e., the spectrograms of \(f\) and \(h\) with respect to the window function \(\varphi\) agree everywhere on \(\mathbb{R}^{d'}\), implies that \(f \sim h\).
\end{proof}

The prototypical example of a Hermite function is the standard Gaussian \( \mathfrak{h}_{0} \), given by:  
\[
\mathfrak{h}_{0}(t) = 2^{d/4} e^{-\pi\|t\|_{2}^{2}},
\]  
where \(\|t\|_{2}^{2} = \sum_{j=1}^{d'} t_{j}^{2}\). Any other Hermite function is a product of \( \mathfrak{h}_{0} \) with a polynomial. Choosing Hermite functions as window functions leads to the following concise statement:

\begin{cor}[Combined result for Gaussian-type and Hermite function]
Let $\gamma > 0$ and let \(\varphi\) be a window function of either form: 
\begin{enumerate}
  \item [(1)] \(\varphi(x) = p(x)e^{-\gamma\|x\|_2^2}\); where \(p: \mathbb{C}^{d'}\rightarrow \mathbb{C}\) be an entire function of exponential type;
  \item [(2)] \(\varphi \in L^{2}(\mathbb{R}^{d'})\) be an artitary Hermite function.
\end{enumerate}
If the sampling parameters satisfy
\begin{equation*}
    \alpha < \sqrt{\frac{1}{2\gamma e}}, \quad \beta < b\sqrt{\frac{2\gamma}{e}},
\end{equation*}
then for all \(f, h \in L^2(\mathbb{R}^{d'})\), the following are equivalent:
\begin{enumerate}
    \item \(|V_\varphi f(\lambda)| = |V_\varphi h(\lambda)|\) for every \(\lambda \in \alpha(\sqrt{\mathbb{Z}})^{d'} \times \beta(\sqrt{\mathbb{Z}})^{d'}\);
    \item \(f \sim h\).
\end{enumerate}
Specifically for Hermite functions (case (ii)), the choice \(\gamma = \frac{1}{2b}\) yields the simplified condition \(\alpha < \sqrt{\frac{b}{e}}\).
\end{cor}

\section{STLCT phase retrieval for lattices and band-limited functions}\ 

In this section, we extend the results of  Alaifari and Wellershoff to the case of STLCT phase retrieval. Specifically, we show that for any lattice, functions in \(L^{2}(\mathbb{R})\) can be constructed that do not agree with the global phase but whose Gabor LCT magnitudes sampled in the lattice agree. These functions have good concentration in both time and frequency and can be constructed to be real-valued for rectangular lattices. 
Moreover, we show that for functions in general bandlimited function spaces \(L^{p}([-B, B])\),  STLCT with the Gaussian window can still be accomplished on a uniform lattice. Additionally, we complete the proof of Proposition 3.2 in subsection 4.2.

\subsection{Counterexamples and phase retrieval for band-limited functions}\

For any given $m$, we can find $M_u$ such that there exists function \(f_{\pm}\)  for which the \(\mathcal{G}^{(A)}\) fails to realize phase retrieval on \(m\mathbb{Z}\times \mathbb{R}\) with the support of the following lemma. 

\begin{lem}
(\cite{WLCT} Covariance property)
\[
V_g^{(A)}\left(M_{u}^{(A)}T_{\tau}f\right)(x, \mu) = \sqrt{\frac{1}{ b}}e^{-\frac{id}{2b}(u^2-\mu^{2}) + \frac{i}{b}\tau(u - \mu)}
V_g f\left(x - \tau, \frac{1}{b}(\mu - u)\right),
\]
where \( x, \tau, u, \mu \in \mathbb{R} \). Especially,
\[
V_g^{(A)}\left(M_{u}^{(A)}f\right)(x, \mu) = \sqrt{\frac{1}{ b}}e^{-\frac{id}{2b}(u^2-\mu^{2})}
V_g f\left(x, \frac{1}{b}(\mu - u)\right),
\]
\end{lem}

\begin{thm}
Let $f_{\pm}= (1 \pm i) M_{u}^{(A)}\varphi + (1 \mp i) M_{-u}^{(A)}\varphi\in L^{2}(\mathbb{R})$ with \(\varphi\) be a Gaussian window function, then it holds that $f_{+}$ and $f_{-}$ do not agree up to global phase and yet
\begin{equation*}
 | \mathcal{G}^{(A)} f_{+}|=|\mathcal{G}^{(A)} f_{-}|
\end{equation*}
\end{thm}
\begin{proof}
Consider
\[
f_{\pm} = (1 \pm i) M_{u}^{(A)}\varphi + (1 \mp i) M_{-u}^{(A)}\varphi
\]
where the Gaussian window \(\varphi(t) = e^{-t^2}\). Using that the Gaussian is invariant under the Fourier transform, one
may calculate that
\[
\begin{aligned}
V_{\varphi}\varphi(x, \omega) &= \int_{\mathbb{R}} e^{-t^2} e^{-(t - x)^2} e^{-i t \omega} dt \\
&=\int_{\mathbb{R}} e^{-(t + \frac{x}{2})^2} e^{-(t - \frac{x}{2})^2} e^{-i(t + \frac{x}{2})\omega} dt \\
&= e^{-\frac{ix \omega}{2}} e^{-\frac{x^2}{2}} \int_{\mathbb{R}} e^{-2t^2} e^{-i t \omega} dt \\
&=e^{-\frac{ix \omega}{2} } e^{-\frac{x^2}{2}} \sqrt{\frac{\pi}{2}} e^{-\frac{\omega^2}{8}}\\
&=\frac{1}{2}e^{-\frac{ix \omega}{2}}e^{-\frac{1}{2}(x^{2}+\frac{1}{4}\omega^{2})}.
\end{aligned}
\]

By the linearity and the covariance property of the STLCT, we obtain
\begin{align*}
\mathcal{G}^{(A)} f_{\pm}(x, \mu) 
&= (1 \pm i)\mathcal{G}^{(A)}(M_{u}^{(A)} \varphi)(x, \mu) + (1 \mp i)\mathcal{G}^{(A)}(M_{-u}^{(A)} \varphi)(x, \mu) \\
&= \frac{1}{\sqrt{ b}} e^{-\frac{id}{2 b} (u^2 - \mu^2)} \bigg[  (1 \pm i)V_{\varphi} \varphi\left(x, \frac{1}{b}(\mu - u)\right) \\
&\quad + (1 \mp i)V_{\varphi} \varphi\left(x, \frac{1}{b}(\mu + u)\right) 
\bigg] \\
&= \frac{1}{2\sqrt{ b}} e^{-\frac{id}{2 b} (u^2 - \mu^2)} \bigg[ 
(1 \pm i)e^{-ix\frac{\mu-u}{2b}}e^{-\frac{1}{2}(x^{2}+\frac{(\mu-u)^{2}}{4b^{2}})} \\
&\quad + (1 \mp i)e^{-ix\frac{u+\mu}{2b}}e^{-\frac{1}{2}(x^{2}+\frac{(u+\mu)^{2}}{4b^{2}})} 
\bigg] \\
&= \frac{1}{2\sqrt{ b}} e^{-\frac{id}{2 b} (u^2 - \mu^2)} 
 e^{-ix\frac{\mu}{2b} - \frac{1}{2}x^{2} - \frac{u^{2}+\mu^{2}}{8b^{2}}} \bigg[ 
(1 \pm i)e^{ix\frac{u}{2b}}e^{\frac{u\mu}{4b^{2}}} \\
&\quad + (1 \mp i)e^{-ix\frac{u}{2b}}e^{-\frac{u\mu}{4b^{2}}} 
\bigg]
\end{align*}

If $x=\frac{\pi b}{u}k$, where $k\in \mathbb{Z}$, then it holds that
\begin{align*}
& (1 + i)e^{ix\frac{u}{2b}+\frac{u\mu}{4b^{2}}}+(1- i)e^{-ix\frac{u}{2b}-\frac{u\mu}{4b^{2}}}\\
&= (1+i)e^{\frac{u}{2b}(ix+\frac{\mu}{2b})}+(1- i)e^{-\frac{u}{2b}(ix+\frac{\mu}{2b})}\\
&=(1+i)e^{\frac{u\mu}{4b^{2}}}i^{k}+(1- i)e^{-\frac{u\mu}{4b^{2}}}(-i)^{k}\\
&=\overline{(1-i)e^{\frac{u\mu}{4b^{2}}}(-i)^{k}+(1- i)e^{-\frac{u\mu}{4b^{2}}}i^{k}}\\
&=(-1)^{k}\overline{(1-i)e^{\frac{u\mu}{4b^{2}}}i^{k}+(1- i)e^{-\frac{u\mu}{4b^{2}}}(-i)^{k}}\\
&=(-1)^{k}\overline{(1 -i)e^{ix\frac{u}{2b}+\frac{u\mu}{4b^{2}}}+(1+ i)e^{-ix\frac{u}{2b}-\frac{u\mu}{4b^{2}}}}.
\end{align*}
It follows that
\begin{equation*}
 |\mathcal{G}^{(A)}f_{+}(x, \mu) |=| \mathcal{G}^{(A)}f_{-}(x, \mu) |
\end{equation*}
for all $x\in \frac{\pi b}{u}\mathbb{Z}$ and $\mu\in\mathbb{R}$.
\end{proof}

\begin{figure}
  \centering
  \includegraphics[width=9cm]{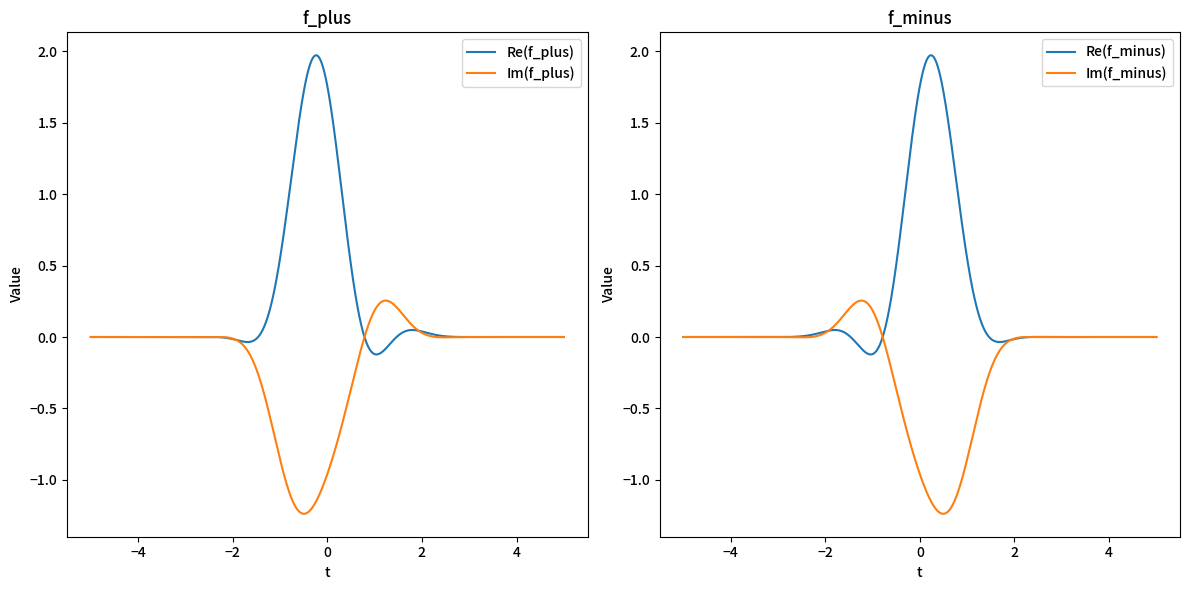}\\
  \begin{center}
  Fig. \(f_{\pm} = (1 \pm i) M_{u}^{(A)}\varphi + (1 \mp i) M_{-u}^{(A)}\varphi\) with \(A=(1, 1, 0, 1)\)
  \end{center}
\end{figure}

However, if we turn to focus on functions in general band-limited function spaces \(L^{p}([-B, B])\),  STLCT with the Gaussian window can still be accomplished on a uniform lattice.

\begin{thm}
Let \(p \in [1, \infty)\), \(B > 0\), and \(m \in \left(0, \frac{1}{4B}\right)\). Then, the following are equivalent for \(f, g \in L^{p}([-B, B])\):  
\begin{enumerate}
    \item [(1)] \(f = e^{i\alpha} g\) for some \(\alpha \in \mathbb{R}\); 
    \item [(2)] \(|\mathcal{G}^{(A)}f| = |\mathcal{G}^{(A)}g|\) on \(\mathbb{N} \times mb\mathbb{Z}\).
\end{enumerate}   
\end{thm}

\begin{proof}

 By the definition of STLCT and its connection to STFT:
\[
V_g^{(A)} f(x, \mu) = \frac{1}{\sqrt{b}}\exp\left(i\frac{d}{2b}\mu^2\right) V_g\left( \tilde{f}\right)\left(x, \frac{1}{b}\mu\right),
\]
where \( \tilde{f}(t) = f(t) e^{i\frac{a}{2b}t^2} \). Due to the analyticity of the Gaussian window \( \varphi \), the STLCT magnitude squared becomes:
 \[
|V_\varphi^{(A)} f(x, \mu)|^2 = b\left| V_\varphi \tilde{f}\left(x, \frac{1}{b}\mu\right) \right|^2.
\] The original STFT sampling lattice in Theorem 1.1 of \cite{W} is \( \mathbb{N} \times m \mathbb{Z} \). For STLCT, the frequency axis scaling \( \frac{1}{b}\mu \) preserves the lattice structure, maintaining consistency with the STFT sampling density.

Since \( \tilde{f} \) and \( \tilde{h} \) remain belong to 
\(L^{p}(-B, B)\) and STFT phase retrieval uniqueness holds for lattice sampling (original Theorem 1.1), we obtain:
\[
 \left| V_\varphi \tilde{f}\left(x, \frac{1}{b}\mu\right) \right| = \left| V_\varphi \tilde{h}\left(x, \frac{1}{b}\mu\right) \right| \quad \forall (x, \mu) \in \mathbb{N} \times mb\mathbb{Z}.
   \]
   By the uniqueness conclusion of the original theorem, there exists \( \tau \in \mathbb{T} \) such that \( \tilde{f} = \tau \tilde{h} \), i.e.,
   \[
   f(t) e^{i\frac{a}{2b}t^2} = \tau h(t) e^{i\frac{a}{2b}t^2}.
   \]
   Canceling the common phase factor \( e^{i\frac{a}{2b}t^2} \), we conclude \( f = \tau h \).
\end{proof}

\subsection{Proof of the proposition 3.2}

Before proving this proposition, we first introduce the necessary lemmas and notations.
Recall that a set \( \Lambda \subseteq \mathbb{C}^{\mathbb{d}} \) is said to be a uniqueness set for \( S \subseteq \mathcal{O}(\mathbb{C}^{\mathbb{d}}) \). if for every \( F \in S \) it holds that
\[
(F(\lambda) = 0 \ \forall \lambda \in \Lambda) \implies F = 0.
\] 

The following lemma establishes that any entire function vanishing on a positive measure subset of \(\mathbb{R}^{\mathbb{d}}\) must be identically zero. This extends the classical identity theorem for analytic functions to higher dimensions. 

\begin{lem}(\cite{FCM}, Lemma 2.3)
Let \(\Lambda \subseteq \mathbb{R}^{\mathbb{d}}\) be a Lebesgue-measurable set such that \(L^{\mathbb{d}}(\Lambda) > 0\), where \(L^{\mathbb{d}}\) denotes the \(\mathbb{d}\)-dimensional Lebesgue measure. Then \(\Lambda\) is a uniqueness set for \( \mathcal{O}(\mathbb{C}^{\mathbb{d}})\).
\end{lem}

\begin{lem}(\cite{FCM}, Lemma 3.2)
Let \(F: \mathbb{R}^{\mathbb{d}} \times \mathbb{C}^{\mathbb{d}} \to \mathbb{C}\) be a function subjected to the following assumptions:
\begin{enumerate}
    \item \(F(\cdot, z)\) is Lebesgue-measurable for every \(z \in \mathbb{C}^{\mathbb{d}}\);
    \item \(F(t, \cdot)\) is an entire function for every \(t \in \mathbb{R}^{\mathbb{d}}\);
    \item The function \(z \mapsto \int_{\mathbb{R}^{\mathbb{d}}} |F(t, z)| \, dt\) is locally bounded, i.e., for every \(z_{0} \in \mathbb{C}^{\mathbb{d}}\), there exists \(\delta > 0\) such that 
    \[
    \sup_{\substack{z \in \mathbb{C}^{\mathbb{d}} \\ |z - z_{0}| \leq \delta}} \int_{\mathbb{R}^{\mathbb{d}}} |F(t, z)| \, dt < \infty.
    \]
\end{enumerate}
Then \(z \mapsto \int_{\mathbb{R}^{\mathbb{d}}} F(t, z) \, dt\) defines an entire function of \(\mathbb{d}\) complex variables.
\end{lem}

By using proposition 3.1, we gain that $L^2(\R)$ is phase retrievable under the short-time linear canonical transform sampling on $\R$ as stated in the following proposition. 
\begin{proof}
Consider the absolute value of the STLCT of a function $(f \in L^{2}(\mathbb{R}^{d'})$ with respect to the window function 
$\varphi$ in the second argument, i.e., the map $(\mu \mapsto|V_{\varphi}^{(A)} f(x, \mu)|^{2})$ for some fixed $(x \in \mathbb{R}^{d'}).$ Using the definition of the STLCT as the LCT of the product of \(f\) with a shift of $\varphi$ yields the identity, $(B=(-d, b, -\frac{1}{b}, 0)),\ (C=(0, b, -\frac{1}{b}, d)).$
\begin{align*}
|V_\varphi^{(A)}(x, \mu)|^2 &= L_B \left( L_A f \cdot T_\mu \overline{L_C \varphi} \right)(d\mu-x) 
\overline{ L_B \left( L_A f \cdot T_\mu \overline{L_C \varphi} \right)(d\mu-x)}\\
&=\left[L_B \left( L_A f \cdot T_\mu \overline{L_C \varphi} \right) 
\mathcal{R}L_B \left( \overline{L_A f} \cdot (T_\mu L_C \varphi) \cdot e^{\frac{id}{b}t^{2}}\right)\right](d\mu-x) \\
&=\left[L_B \left( L_A f \cdot T_\mu \overline{L_C \varphi} \right)
L_B \mathcal{R}\left( \overline{L_A f} \cdot (T_\mu L_C \varphi) \cdot e^{\frac{id}{b}t^{2}}\right)\right](d\mu-x) \\ 
\end{align*}
Let \(F_{\mu}=L_A f \cdot T_\mu \overline{L_C \varphi}\), and \( \overline{L_A (f)(u)} =e^{-\frac{id}{b}u^{2}}\mathcal{R}L_{A}(\overline{f}e^{-\frac{ia}{b}t^{2}})(u).
\)
Thus, the above equality yields to 
\begin{equation*}
|V_\varphi^{(A)}(x, \mu)|^2 =L_B\left( \mathcal{R}(\overline{F_{\mu}}e^{\frac{id}{b}t^{2}})  *_B F_{\mu}\right)(d\mu-x)
\end{equation*}
The convolution is given by
\begin{equation*}
\left(\mathcal{R}(\overline{F_{\mu}}e^{\frac{id}{b}t^{2}}) *_B F_{\mu}\right)(t)=
\frac{\overline{\lambda}_{B}(t)}{\sqrt{2\pi b}} \left(\mathcal{R}(\overline{F_{\mu}}e^{\frac{id}{b}t^{2}})e^{-\frac{i d}{2 b}t^{2}}\right) * \left( F_{\mu}e^{-\frac{i d}{2 b} t^{2}} \right)\\
\end{equation*}
we have 
\begin{align*}
&\left(\mathcal{R}(\overline{F_{\mu}}e^{\frac{id}{b}t^{2}})e^{-\frac{i d}{2 b}t^{2}}\right) * \left( F_{\mu}e^{-\frac{i d}{2 b} t^{2}} \right)(s) \\
& = \int_{\mathbb{R}^{d'}}\overline{L_A f(-(s-u))}e^{\frac{i d}{2 b} t^{2}}
L_C \varphi(-(s-u)-\mu)L_A f(u)\overline{L_C \varphi(u-\mu)}e^{-\frac{i d}{2 b} t^{2}}du\\
&=\int_{\mathbb{R}^{d'}}F_{s}(u)\overline{T_{\mu}\Phi_{s}(u)} du\\
&=\langle F^{A}_{s}, T_{\mu}\Phi^{C}_{s} \rangle,
\end{align*}
where $F^{A}_{s}=L_A f T_{s}\overline{L_A f}$, $\Phi^{C}_{s}=L_C \varphi T_{s}\overline{L_C \varphi}$.

Hence, if \(|V^{(A)}_{\varphi} f(x, \mu)| = |V^{(A)}_{\varphi} h(x, \mu)|\), then
\[
\langle F^{A}_{s}, T_{\mu}\Phi^{C}_{s} \rangle=\langle H^{A}_{s}, T_{\mu}\Phi^{C}_{s} \rangle
\]
It follows that for every fixed $s\in \mathbb{R}$
\[
\langle F^{A}_{s}-H^{A}_{s}, T_{\mu}\Phi^{C}_{s} \rangle=(F^{A}_{s}-H^{A}_{s})\ast(\mathcal{R}(\Phi^{C}_{s}))(\mu)=0
\]
Applying the Fourier transform yields the relation
\[
\mathcal{F}((F^{A}_{s}-H^{A}_{s})\ast(\mathcal{R}(\Phi^{C}_{s})))=\mathcal{F}(F^{A}_{s}-H^{A}_{s})(\mu)
\mathcal{F}(\mathcal{R}(\Phi^{C}_{s}))(\mu)=0
\]
which holds for every $x\in \mathbb{R}^{d'}$ . Observe that the function $\mathcal{R}(\Phi^{C}_{s})$ has Gaussian decay since $g\neq0$ does not vanish identically. Applying the Lemma 3.2 of \cite{FCM},
The map \(\mu \mapsto \mathcal{F}(\mathcal{R}(\Phi^{C}_{s}))(\mu)\) extends from \(\mathbb{R}^{d'}\) to an entire function which does not vanish identically. In particular, the set 
\[
Z=\left\{\mu \in \mathbb{R}^{d'}: \mathcal{F}\left(\mathcal{R}\left(\Phi^{C}_{s}\right)\right)(\mu)=0\right\}
\]
has \(d'\)-dimensional Lebesgue measure zero (\cite{FCM} Lemma 2.3), otherwise, 
\(\mathcal{F}\left(\mathcal{R}\left(\Phi^{C}_{s}\right)
\right)\) must vanish identically. Since \(\mathcal{F}(F^{A}_{s}-H^{A}_{s})\) is continuous, it follows that \(\mathcal{F}(F^{A}_{s}-H^{A}_{s})\) must vanish identically. The arbitrariness of \(\mu \in \mathbb{R}^{d'}\) shows that \(\mathcal{F}(F^{A}_{s}-H^{A}_{s})\) vanishes identically for every \(\mu \in \mathbb{R}^{d'}\). Then, for any $s$
\[
L_A f T_{s}\overline{L_A f}=L_A h T_{s}\overline{L_A h} 
\]
holds for almost everywhere.
Let $s=0$, the above implies $L_A f\sim L_A h$, which is equivalent to $f\sim h$. 
\end{proof}

\section*{Statements and Declarations}
The author states that there is no conflict of interest. No datasets were generated or analyzed during the current study.

\section*{Acknowledgments}
The first author gratefully acknowledges the support of the Chern Institute of Mathematics Visiting Scholar Program at Nankai University during the completion of this work.

\end{document}